\documentclass{amsart}
\pdfoutput=1

\usepackage{url}
\usepackage{microtype}
\usepackage{tikz}
\usepackage[pdftex,colorlinks,citecolor=black,linkcolor=black,urlcolor=black,bookmarks=false]{hyperref}
\def\MR#1{\href{http://www.ams.org/mathscinet-getitem?mr=#1}{MR#1}}
\def\arXiv#1{arXiv:\href{http://arXiv.org/abs/#1}{#1}}
\usepackage{doi}
\usepackage{booktabs}

\theoremstyle{plain}
\newtheorem{theorem}{Theorem}

\newtheorem{lemma}[theorem]{Lemma}
\newtheorem{proposition}[theorem]{Proposition}
\newtheorem{conjecture}[theorem]{Conjecture}

\numberwithin{theorem}{section}

\numberwithin{equation}{section}

\newcommand{\R}{{\mathbb R}}
\newcommand{\Q}{{\mathbb Q}}
\newcommand{\Z}{{\mathbb Z}}
\newcommand{\C}{{\mathbb C}}
\newcommand{\N}{{\mathbb N}}
\newcommand{\T}{{\mathcal{T}}}
\newcommand{\vol}{\mathop{\textup{vol}}}

\title[Properties of optimal functions]{Some properties
of optimal functions\\
for sphere packing in dimensions $8$ and $24$}

\author{Henry Cohn}
\address{Microsoft Research New England\\
One Memorial Drive\\
Cambridge, MA 02142} \email{cohn@microsoft.com}

\author{Stephen D.\ Miller}
\address{Department of Mathematics\\
Hill Center--Busch Campus\\
Rutgers\\
110 Frelinghuysen Rd.\\
Piscataway, NJ 08854-8019} \email{miller@math.rutgers.edu}

\thanks{Miller was partially supported by NSF grant
 DMS-1201362 and an Alfred P.\ Sloan Foundation Fellowship.}

\date{March 15, 2016}

\begin{document}

\begin{abstract}
We study some sequences of functions of one real variable and
conjecture that they converge uniformly to functions with certain
positivity and growth properties.  Our conjectures imply a
conjecture of Cohn and Elkies, which in turn implies the complete
solution to the sphere packing problem in dimensions $8$ and $24$.
We give numerical evidence for these conjectures as well as some
arithmetic properties of the hypothetical limiting functions.
The conjectures are of greatest interest in dimension $24$, in light of
Viazovska's recent solution to the Cohn-Elkies conjecture (and
consequently the sphere packing problem) in dimension $8$.
\end{abstract}

\maketitle

\section{Introduction}

One of the fundamental problems in geometry is to determine the densest
sphere packing in Euclidean space.  In other words, how large a fraction of
$\R^n$ can be covered by equal-sized, non-overlapping balls?  The answer is
known so far only for $n \le 3$ (see \cite{Toth} and \cite{Ha}),
and very recently for $n=8$ as well \cite{V}.  A
remarkable feature of this problem is that each dimension has its own
idiosyncrasies. Even setting aside the issue of proofs, the best packings
known do not seem to follow any simple pattern.

Perhaps the most striking packings are those formed by centering spheres at
the points of the $E_8$ root lattice and the Leech lattice. Both have been known
for some years now to be the densest lattice packings in their dimensions. The $E_8$
case was proved by Blichfeldt in his 1935 paper \cite{Bl}, and the Leech
lattice case was proved by Cohn and Kumar in \cite{CK3} (see also
\cite{CK1}). The latter work was based on an analytic approach introduced in
\cite{CE} by Cohn and Elkies, who in fact studied the general sphere packing
problem (including non-lattice packings, which may improve on the density of
lattice packings in some dimensions). Cohn and Elkies proved that $E_8$ and
the Leech lattice are optimal among all sphere packings if there exist functions
from $\R$ to $\R$ satisfying certain sign and regularity conditions; they
furthermore conjectured that such functions do indeed exist. In this paper we
introduce explicit sequences of functions which we conjecture converge to
functions satisfying the Cohn-Elkies conditions.
(We of course note that the $n=8$ case of the Cohn-Elkies conjecture was solved in \cite{V}.)

Our functions depend on a parameter $n$, the dimension of the sphere packing
problem. One advantage of our approach is that our conjectures appear to hold
for a broader range of values of $n$, not only for $n=8$ and $n=24$. Although
they have no sphere packing implications except in those two cases, existence
might be easier to prove because they no longer depend on delicate facts
about these particular dimensions.

A second advantage is that our approach does not rely on numerical
optimization.  By contrast, the Leech lattice optimality proof makes use of a
carefully optimized polynomial of degree $803$ with $3000$-digit
coefficients. The computer-assisted proof in \cite{CK3} reads this polynomial
from a file and verifies that it has the desired properties to complete the
proof, but there is no conceptual description of the polynomial or simple
method to construct it from scratch. (It was found by combining numerous \emph{ad
hoc} techniques to locate a starting point from which Newton's method would
converge.)  Using our approach, one could replace this complicated polynomial
with a polynomial that has a much simpler description.  That would not remove
the need for computer verification of its properties, but it is a step
towards simplifying the proof.

Our lack of need for optimization also enables us to carry out much larger
computations than in previous papers.  For example, we arrive at density
bounds that are sharp to over fifty decimal places in $\R^8$ and $\R^{24}$,
compared with the fourteen and twenty-nine decimal places from \cite{CK3}.
Strictly speaking our new bounds are not theorems, because we have not
bothered to verify them using exact arithmetic, but our floating point
calculations leave no reasonable doubt. We are confident that the approach
from Appendix~A in \cite{CK3} could be used to provide a proof (should a rigorous bound be
needed for some purpose).

The results of these large calculations display intricate and surprising
structure.  Most interestingly, in Section~\ref{sec:rat} we find that the
second Taylor coefficients appear to be rational. If the pattern governing
the higher coefficients could be identified, it would yield a direct
construction by power series of functions satisfying the Cohn-Elkies
conjecture.

In the next section we review background from \cite{CE}. Our functions are
introduced in Section~\ref{sec:conj}. In Section~\ref{sec:num}, we provide
experimental evidence that our sequences of functions are converging rapidly
(despite the failure of a related, naive construction), and we study this numerical data in detail.
In Section~\ref{sec:rat} we examine the Taylor coefficients and values of the
Mellin transform of the optimal functions, both of which exhibit some unexplained rationality properties. In Section~\ref{sec:energy} we study the closely
related problem of potential energy minimization. Finally we conclude in
Section~\ref{sec:single} by describing some related but simpler sequences of
functions, which serve as a testing ground for our main conjectures.

\section{Background}

Define the Fourier transform of a function $f \colon \R^n \to \R$
by
\begin{equation}\label{fhat}
\widehat{f}(t) = \int_{\R^n}   f(x)   e^{- 2  \pi i  \langle x,t
\rangle} \, dx.
\end{equation}
We call a continuous function $f$ \emph{admissible} if both
$|f(x)|$ and $|\widehat{f}(x)|$ are bounded by a constant times
$(1+|x|)^{-n-\delta}$ for some $\delta>0$.  This bound ensures,
for example, that the integral defining $\widehat{f}$ converges.
It also guarantees that both sides of the Poisson summation
formula
\[
\sum_{x  \in  \Lambda}  f(x) = \frac{1}{|\Lambda|} \sum_{t \in
\Lambda^*}  \widehat{f}(t)
\]
converge absolutely and are equal.  Here $\Lambda$ denotes a lattice in
$\R^n$,  $\Lambda^* = \{t \in \R^n : \langle t,x \rangle \in \Z
\text{~for all~}x\in \Lambda \}$ its dual, and $|\Lambda| =
\vol(\R^n/\Lambda)$ its covolume.

Our primary connection between sphere packing and Fourier analysis
is the following theorem of Cohn and Elkies (Theorem 3.1 in
\cite{CE}; see also \cite{Co}):

\begin{theorem} \label{theorem:cohnelkies}
Suppose there exists  an admissible function $f \colon \R^n \to
\R$ and a constant $r$ such that
\begin{enumerate}
\item \label{cond1} $f(0)  =  \widehat{f}(0)  \ne  0$,

\item \label{cond2} $f(x) \le 0   $ for $ |x| \ge r ,  $ and

\item \label{cond3} $\widehat{f}(t)  \ge  0$ for all $t$.
\end{enumerate}
Then every sphere packing in $\R^n$ has density at most
\[
\frac{\pi^{n/2}}{(n/2)!}\left(\frac{r}{2}\right)^n.
\]
\end{theorem}

As usual $(n/2)!$ is to be interpreted as $\Gamma(n/2+1)$ when $n$
is odd.  The \emph{density} of a sphere packing refers to the
fraction of space covered by the packing.

We will briefly explain how to prove Theorem~\ref{theorem:cohnelkies}
using Poisson summation, because the conditions for a sharp bound
will be important later in the paper.

\begin{proof}
First, we give the proof for lattice packings, after which we will
sketch the general proof.

Suppose $\Lambda \subset \R^n$ is a lattice.  We can assume
without loss of generality that the minimal nonzero vector length
in $\Lambda$ is $r$, because sphere packing density is invariant
under scaling.  That amounts to using balls of radius $r/2$ in the
sphere packing.

By Poisson summation,
$$
\sum_{x  \in  \Lambda}  f(x) = \frac{1}{|\Lambda|} \sum_{t \in
\Lambda^*}  \widehat{f}(t).
$$
Applying the inequalities on $f$ and $\widehat{f}$ yields
$$
f(0) \ge \sum_{x  \in  \Lambda}  f(x) = \frac{1}{|\Lambda|}
\sum_{t \in \Lambda^*}  \widehat{f}(t) \ge
\frac{\widehat{f}(0)}{|\Lambda|}.
$$
Thus,
$$
|\Lambda| \ge 1.
$$
In other words, there is at most one lattice point per unit volume
in $\R^n$.  It follows that the density is at most the volume of a
sphere of radius $r/2$, i.e.,
$$
\frac{\pi^{n/2}}{(n/2)!}\left(\frac{r}{2}\right)^n
$$
(because the density equals the volume of a sphere times the
number of spheres per unit volume in space).

For the general case, one can assume without loss of generality
that the sphere packing is periodic, i.e., a union of translates
of a lattice packing.  Suppose it is the disjoint union of
$\Lambda+v_1,\dots,\Lambda+v_n$.  Then applying the identity
$$
\sum_{j,k  =  1}^N \sum_{x\in\Lambda} f(x+v_j-v_k) =
\frac{1}{|\Lambda|}\sum_{t\in\Lambda^*}
    \widehat{f}(t)\left|\sum_{j=1}^Ne^{2\pi i \langle v_j,t \rangle}
    \right|^{2},
$$
which follows from Poisson summation and some manipulation,
completes the proof as above.
\end{proof}

One can weaken the hypothesis of admissibility in this theorem, at the cost
of complicating the proof (see Proposition~9.3 in \cite{CK2}, which is set in
the more general context of potential energy minimization, or the proof in
\cite{CZ2}, which does not even use Poisson summation).  However, the applications
in this paper will use only admissible functions.

Unfortunately, Theorem~\ref{theorem:cohnelkies} does not address the
issue of how to find functions $f$ that lead to good sphere packing
bounds (i.e., that minimize $r$). Doing so amounts to grappling with an
infinite-dimensional optimization problem, which has a simple solution
when $n=1$ but is unsolved and appears difficult for $n>1$. Cohn and
Elkies performed a computer search to locate explicit functions that
improve on the previously known density upper bounds for $4 \le n \le
36$. (For $4 \le n \le 7$ and $n=9$, a refinement of this approach
from \cite{LOV12} yields slightly better bounds.)
These functions are probably nearly optimal in terms of minimizing
the values $r$ achieved by functions satisfying the hypotheses of Theorem~\ref{theorem:cohnelkies}.
However, in most cases these bounds are still far above the densities
of the best packings known.

The most remarkable application of Theorem~\ref{theorem:cohnelkies} occurs
when the dimension $n$ is $8$ or $24$.  In those dimensions, Cohn and Elkies
found functions that come tantalizingly close to solving the sphere packing
problem completely. Using more sophisticated search techniques, Cohn and
Kumar \cite{CK3} later achieved a bound within a factor of $1+1.65\times
10^{-30}$ of the conjectured optimum for $n=24$ and a factor of $1+10^{-14}$
for $n=8$.  Typically it is harder to get more accurate bounds for larger values
of $n$; the reason the bound for $n=24$ is so much better is that Cohn and Kumar
required that level of accuracy for their application and thus devoted much more
computer time to optimizing this case.

One may ask whether the functions produced by these computer searches
asymptotically produce a sharp sphere packing bound in these dimensions. That
appears to be true, and Cohn and Elkies conjectured an even stronger
statement, namely that the sphere packing problem in dimensions $2$, $8$, and
$24$ can be solved exactly by the use of a single function $f$ in
Theorem~\ref{theorem:cohnelkies}:

\begin{conjecture}[Conjecture~7.3 in \cite{CE}; now a theorem when $n=8$ \cite{V}] \label{conjecture:cohnelkies}
When $n \in \{2,8,24\}$, there exists a function $f$ satisfying
the hypotheses of Theorem~\ref{theorem:cohnelkies} with
\[
r=
\begin{cases}
(4/3)^{1/4} & \textup{if $n=2$,}\\
\sqrt{2} & \textup{if $n=8$, and}\\
2 & \textup{if $n=24$.}\\
\end{cases}
\]
\end{conjecture}

The sphere packing problem is of course trivial for $\R^1$, where
\[
f(x) = \frac{1}{1-x^2}\left(\frac{\sin \pi x}{\pi x}\right)^2
\]
gives an optimal function for use in Theorem~\ref{theorem:cohnelkies}.
At first glance it may seem quite unlikely that
Theorem~\ref{theorem:cohnelkies} leads to a sharp sphere packing bound in any
other dimension $n>1$.  For example, positivity arguments such as its proof
(which involve dropping a number of terms to get an inequality) nearly always
lose information; in analytic number theory it is essentially a given that they
will not produce sharp results.

Despite this, there is ample numerical evidence that Conjecture~\ref{conjecture:cohnelkies}
is true in the special dimensions $n=2$, $8$ (where it was proved in \cite{V}), and $24$.
Similarly sharp solutions have been found for related problems in $\R^2$,
$\R^8$, and $\R^{24}$ such as the kissing problem (see \cite{L,OS}), and there are many
analogies with error-correcting codes (see, for example, \cite{CZ1}).

The main purpose of this paper is to introduce explicit sequences which we
conjecture converge to functions satisfying
Conjecture~\ref{conjecture:cohnelkies}. We will focus on $n=8$ and $24$, not
only because these cases are more interesting, but also because they appear
to be more similar to each other than either is to the $n=2$ case.

\section{Explicit functions}
\label{sec:conj}

The conditions on $f$ in Theorem~\ref{theorem:cohnelkies} are
radially symmetric, so any function satisfying them can be
rotationally symmetrized.  Thus, without loss of generality we
will assume that $f$ is a radial function, and we will
sometimes write $f(r)$ for the common value $f(x)$ with
$|x|=r$. A convenient family of functions to consider are
products of polynomials with Gaussians. If we write
\begin{equation}\label{pf}
f(x)  =    p(|x|^2) e^{-\pi|x|^2}
\end{equation}
with $p$ a polynomial, then a calculation shows
\[
\widehat{f}(t) =  (\T p)(|t|^2) e^{-\pi|t|^2}
\]
for some polynomial $\T p$ depending on $p$.  In other words, $\T$ is
the linear map given by
\begin{equation}\label{Tdef}
(\T p)(|t|^2) = e^{\pi|t|^2}  \int_{\R^n} p(|x|^2) e^{-\pi|x|^2}
e^{-2\pi i \langle x,t \rangle} \, dx,
\end{equation}
which one can check maps polynomials to polynomials.

The functions $f$ used in \cite{CE} are of the form \eqref{pf}.
They are created by requiring that
\[
f(0) = \widehat{f}(0) = 1
\]
and also that $f$ and $\widehat{f}$ must have forced single and
double roots at certain locations.  Together these can be
interpreted as a set of linear conditions satisfied by the
coefficients of the polynomial $p$, which can be solved when the
degree of $p$ is appropriately large compared to the number of
forced roots.   Cohn and Elkies used a computer search to choose
locations for these forced roots in order to optimize the sphere
packing bound obtained from Theorem~\ref{theorem:cohnelkies}.

This procedure works well in practice, but it is difficult
to analyze.  It is not at all obvious that these successively optimized
functions $f$ (coming from polynomials of higher and higher degree) even
converge to a locally optimal choice of $f$, let alone the global optimum.
The numerical evidence is compelling, but a proof is completely lacking.

In this paper, we examine a simpler variant of this approach.
Instead of carefully optimizing the forced root locations, we
specify them \emph{a priori}.  Specifying the roots is worse in
practice, but not much worse: for example, using $200$ roots we will come
within a factor of $1+1.23\times 10^{-27}$ of the Leech lattice's density,
compared with $1+1.65\times10^{-30}$ in \cite{CK3} using $200$ carefully
optimized roots. Because our functions are explicit and do not involve
a computer search, they can be computed more quickly and
may be easier to analyze.

In order to describe where and why we force roots of $f$ and
$\widehat{f}$, it is helpful to recall the proof of
Theorem~\ref{theorem:cohnelkies}.  There we proved the inequality
\[
f(0) \ge \sum_{x  \in  \Lambda}  f(x) = \frac{1}{|\Lambda|}
\sum_{t \in \Lambda^*}  \widehat{f}(t) \ge
\frac{\widehat{f}(0)}{|\Lambda|}
\]
using the conditions that $f(x) \le 0$ for $|x| \ge r$ and
$\widehat{f}(t) \ge 0$ for all $t$.  If the lattice $\Lambda$ is
actually the densest sphere packing in $\R^n$, and if this method
proves a sharp bound, then both inequalities must actually be
equalities. For that to happen, one must first have $|\Lambda|=1$.
(Recall that in the proof, we scaled $\Lambda$ so its minimal vector
length is $r$.) For this scaling of $\Lambda$, the terms $f(x)$ and
$\widehat{f}(t)$ must vanish whenever $x \in \Lambda_{\neq 0}$
and $t \in \Lambda^*_{\neq 0}$. In other words,
\begin{equation}\label{fandfhatmustvanish}
\begin{aligned} & \text{$f$ must vanish at all nonzero vector lengths, and}\\
& \text{$\widehat{f}$ must vanish at all nonzero dual vector
lengths.}
\end{aligned}
\end{equation}
In order to preserve the sign constraints~\eqref{cond2} and~\eqref{cond3}
from Theorem~\ref{theorem:cohnelkies}, the order of vanishing at every vector
length must be even, with the exception of  $f(x)$ at $|x|=r$, where a sign
change should in fact occur.

Note that even if one did not assume that $f$ is radial, it would
still vanish on concentric spheres through the lattice points, not
simply at the individual lattice points.  This is because the above argument
applies not only to $\Lambda$, but to any rotation of it; consequently, $f$ must vanish at each rotated lattice point.

Table~\ref{lengthtable} lists the lengths of nonzero vectors in the
optimal lattices in dimensions $1$, $2$, $8$, and $24$ (scaled so that
$|\Lambda|=1$, which is the usual scaling except in $\R^2$); these
lattices are undoubtedly the densest sphere packings in their
respective dimensions, but of course this has not been proved in $8$ or
$24$ dimensions. For each of these lattices, the dual vector lengths
are the same as the vector lengths: in each case except dimension $2$,
$\Lambda^* = \Lambda$, and in dimension $2$, $\Lambda^*$ is a rotation
of $\Lambda$.

\begin{table}
\caption{Vector lengths in optimal lattices normalized with
$|\Lambda|=1$.} \label{lengthtable}
\begin{center}
\begin{tabular}{ccccc}
\toprule
dimension & \ \ & lattice & \ \ & vector lengths\\    % ad hoc spacing
\midrule
$1$ && $\Z$ && $\{k : k \ge 0\}$\\
$2$ && hexagonal && $\{(4/3)^{1/4}\sqrt{k^2+k\ell+\ell^2} : (k,\ell) \in \Z^2\}$\\
$8$ && $E_8$ && $\{\sqrt{2k} : k \ge 0\}$\\
$24$ && Leech && $\{\sqrt{2k} : k \ge 0, k \ne 1\}$\\
\bottomrule
\end{tabular}
\end{center}
\end{table}

One naive approach to constructing optimal functions would be to force
roots at exactly these locations.  Specifically, let $r_1 < r_2 <
\dots$ be the nonzero vector lengths in the last column of
Table~\ref{lengthtable}. (In other words, $r_1 = \sqrt{2}$ if $n=8$ and
$r_1 =2$ if $n=24$, etc.)  For any integer $k\ge 1$ we define the function
$f_k(x)$ to be of the form \eqref{pf}, with $p(x)=q_k(x)$ a polynomial
of degree $4k-1$, subject to the following $4k$ constraints:
\begin{equation}\label{constraints}
\aligned f_k(0) &    =  1, \\
f_k(x) & \textup{ vanishes to order 1 at }   |x|   =  r_1, \\
f_k(x) & \textup{ vanishes to order 2 at }   |x|   =  r_2,\ldots,r_k, \textup{ and} \\
\widehat{f_k}(x) & \textup{ vanishes to order 2 at } |x| =
r_1,\dots,r_k.
\endaligned
\end{equation}
Such a function is designed to satisfy the requirements of
\eqref{fandfhatmustvanish} and thereby be used in
Theorem~\ref{theorem:cohnelkies}.  However condition~\eqref{cond1} of the
theorem has not been addressed; i.e., we have not forced $\widehat{f_k}(0)=1$
as well.  This condition in fact holds automatically for the limit $f$ of the
functions $f_k$, provided it exists; the reason is that $f$ and $\widehat{f}$
vanish at all non-zero lattice points, and Poisson summation over the lattice
$\Lambda$ implies $\widehat{f}(0)=f(0)=1$.  If one wishes to use the
functions $f_k$ themselves to prove sphere packing bounds, then one must
rescale them to force condition~\eqref{cond1} to hold.  This rescaling
changes the bound to
\[
\frac{\pi^{n/2}}{(n/2)!}\left(\frac{r_1}{2}\right)^n\frac{f_k(0)}{\widehat{f_k}(0)}
\]
(i.e., it introduces a factor of $f_k(0)/\widehat{f}_k(0)$).

Unfortunately, this sequence of functions fails, at first subtly and then dramatically: the functions
do not converge as $k \to \infty$, and for sufficiently large $k$ they do
not even prove packing bounds at all (because they develop unwanted sign changes).
See Section~\ref{sec:num} for a discussion of the numerical evidence.

Instead of using the exact vector lengths in the definition of $f_k$, we
modify them as follows.  Let $\ell_m$ denote the actual $m$-th vector
length.  Given $k$, we define modified root locations $r_1,\dots,r_k$
(depending on $k$) as follows:
\begin{equation} \label{eq:betterrdef}
r_m = \begin{cases}
\ell_m & \textup{if $m < \lfloor 2k/3 \rfloor$, and}\\
\sqrt{\ell_m^2 + \frac{1}{4} \ell_k^2 \left(\frac{m - \lfloor 2k/3 \rfloor}
{k - \lfloor 2k/3 \rfloor}\right)^2} & \textup{if $\lfloor 2k/3 \rfloor \le m \le k$.}
\end{cases}
\end{equation}
In other words, the first two-thirds of the root locations are left
unchanged, while the squares of the others are perturbed by a quadratically
growing amount culminating in making the final one $25\%$ larger.  The
numbers $2/3$ and $1/4$ in \eqref{eq:betterrdef} are somewhat arbitrary, but
these choices appear to work well in practice.  The rescaling \eqref{eq:betterrdef}
was motivated by the empirical location of the roots of the optimized functions of
particular degrees mentioned earlier, as well as the similar spacing of large roots of orthogonal polynomials
(see \cite{deift}).

We can now use these modified root locations to define functions $f_k$.
Unlike the naive definition using $\ell_m$, the improved definition using $r_m$ appears to work well. In
Section~\ref{sec:num} we will examine numerical evidence and make
conjectures, but before that we must resolve one theoretical issue: it is not
obvious that the functions $f_k$ even exist, because the linear equations
defining them may have no solution. In fact, if the forced root locations
$r_1,r_2,\dots$ were chosen differently, then this difficulty could occur.
For example, for $n=1$, $k=2$, $r_1=1$, and $r_2 = 1.3403207576\dots$ (chosen
to satisfy a certain polynomial equation with coefficients in $\Q[\pi]$), the
constraints~\eqref{constraints} defining $f_2$ have no solution. Fortunately,
existence and uniqueness do hold in our cases:

\begin{lemma}\label{unique}
For any algebraic numbers $0 < r_1 <\dots < r_k$, there exists a unique
polynomial $q_k$ of degree $4k-1$ such that the
constraints~\eqref{constraints} hold for $f_k(x) =
q_k(|x|^2)e^{-\pi|x|^2}$.
\end{lemma}

For the proof of this lemma, we will need to diagonalize the
transform $\T$ defined in \eqref{Tdef}.  Define $p_j(x)=L_j^{n/2-1}(2\pi x)$, where $L_j^\alpha$ is
the Laguerre polynomial of degree $j$ and index $\alpha=n/2-1$. Recall
that the polynomials $L^\alpha_j$ are orthogonal polynomials with
respect to the measure $x^{-\alpha}e^{-x} \, dx$ on $[0,\infty)$, which
 can be written as
\[
L^\alpha_j(x)   =
\frac{x^{-\alpha}e^x}{j!}\frac{d^j}{dx^j}(x^{\alpha+j}e^{-x}).
\]
The product $p_j( |x|^2)e^{-\pi |x|^2}$ is a radial eigenfunction
of the Fourier transform \eqref{fhat} with eigenvalue $(-1)^j$. In
other words,
\[
\T p_j =(-1)^jp_j.
\]
Writing an arbitrary polynomial as a linear combination of the
polynomials $p_j$ makes it easy to apply $\T$.

\begin{proof}
Write the polynomial $q_k$ as a linear combination
\[
q_k = \sum_{j=0}^{4k-1} c_j p_j.
\]
The constraints~\eqref{constraints} amount to the following linear
equations in $c_0,\dots,c_{4k-1}$:
\begin{equation}\label{constreqn}
\begin{split}
\sum_{j=0}^{4k-1} c_j p_j(0) &= 1 \\
\sum_{j=0}^{4k-1} c_j p_j(2\pi r_m^2) &= 0  \qquad\textup{for $1 \le m \le k$}\\
\sum_{j=0}^{4k-1} c_j p_j'(2\pi r_m^2) &= 0  \qquad\textup{for $2 \le m \le k$}\\
\sum_{j=0}^{4k-1} (-1)^j c_j p_j(2\pi r_m^2) &= 0  \qquad\textup{for $1 \le m \le k$}\\
\sum_{j=0}^{4k-1} (-1)^j c_j p_j'(2\pi r_m^2) &= 0
\qquad\textup{for $1 \le m \le k$.}
\end{split}
\end{equation}
To prove the lemma, we need only show that the determinant of the $4k \times 4k$
matrix of coefficients is nonzero.  View the coefficients as
polynomials in $\pi$ (recall that $p_j(x)=L_j^{n/2-1}(2\pi x)$, where
$L_j^{n/2-1}$ has coefficients in $\Q$, and that the forced root locations
$r_1,\dots,r_k$ are algebraic).  We will use the transcendence of $\pi$ to prove that
the determinant is nonzero, by identifying its leading coefficient as
a polynomial in $\pi$ and showing that it does not vanish.

Each column of the matrix corresponds to $p_j$ for some $j$, with entries of the
form $p_j(0)$, $p_j(2\pi r_m^2)$, $p_j'(2\pi r_m^2)$, $(-1)^j p_j(2\pi r_m^2)$, and
$(-1)^j p_j'(2\pi r_m^2)$ for suitable values of $m$.
If we write $p_j$ as a linear combination of monomials, then we can expand the
determinant as a corresponding linear combination, with the highest power of $\pi$ coming from
the monomial $x^j$ of highest degree.  Thus, if we can show that the determinant is nonzero
after replacing $p_j(x)$ with $x^j$ for all $j$, then it must have been nonzero
to start with.

This replacement dramatically simplifies the equations, because we can reinterpret
them as describing a more tractable interpolation problem.  The new equations
ask for the coefficients of a polynomial of degree $4k-1$ with the
following constraints.  Its value at $0$ is specified, its value at
$2\pi r_1^2$ is specified, its values and first derivatives at $2\pi
r_2^2,\dots,2\pi r_k^2$ are specified, and its values and first derivatives at
$-2\pi r_1^2,\dots, -2\pi r_k^2$ are specified.  For the negative cases,
note that replacing $p_j(x)$
with $x^j$ transforms $(-1)^j p_j(x)$ into $(-x)^j$.  This interpolation problem is a
special case of Hermite interpolation, and the determinant of the
coefficient matrix is therefore nonzero.  (See Subsection~2.1 of \cite{CK2}
for a review of Hermite interpolation.)

It follows that the coefficient matrix of the original equations
also has a nonzero determinant, so there exists a unique solution.
\end{proof}

\section{Numerical evidence}
\label{sec:num}

In this section  we will examine the numerical evidence for convergence. Our
calculations are based on floating-point arithmetic, with no rigorous bounds
on the rounding error, but we believe all reported digits are correct. (We believe that these calculations could be made rigorous if necessary, for example by using interval arithmetic
or the techniques from Appendix~A in \cite{CK3}.) When using $k$ forced root locations, we carried out all
computations to $8k+75$ digits of precision using PARI/GP.  Experimentation suggests that
$8k+75$ digits is far more precision than is actually needed, but it is
easier to pick an unnecessarily high bound than to calibrate how little
precision we could get away with.

First, consider the naive approach discussed in the previous section, in
which one takes the forced root locations $r_1,\dots,r_k$ to be the first $k$
nonzero vector lengths in the optimal lattice.  Though at first this approach gives good bounds, it subtly reverses course and eventually   fails
completely for large $n$ (see Table~\ref{table:naive}). In the $\R^8$ case, the bound improves as $k$ grows until $k=40$, at
which point it is slightly better than the bound proved in \cite{CE} (and
much better than the previous record bound of
$\approx 1.012$). However, after $k=40$ the bound steadily gets worse. By $k=130$,
the bound would be less than $1$, which is impossible and indicates that the
function must have developed an unwanted sign change by that point. In the
$\R^{24}$ case, the problems are even more dramatic.

\begin{table}
\caption{Supposed upper bounds for packing density using the exact vector
lengths as forced root locations, without checking for unwanted sign changes.
Bounds are expressed as a multiple of the density of the optimal lattice.}
\label{table:naive}
\centering
\begin{tabular}{ccc}
\toprule
$k$ & naive packing bound in $\R^{8}$ & naive packing bound in $\R^{24}$\\
\midrule
$10$ & $1.0001507518\dots$ &  $\phantom{-}1.3706005433\dots$\\
$20$ & $1.0000052091\dots$ &  $\phantom{-}1.1082380574\dots$\\
$30$ & $1.0000013138\dots$ &  $\phantom{-}1.1109658270\dots$\\
$40$ & $1.0000009656\dots$ &  $\phantom{-}1.2417952436\dots$\\
$50$ & $1.0000014330\dots$ &  $\phantom{-}2.1249579472\dots$\\
$60$ & $1.0000035296\dots$ &  ${-}3.7219923464\dots$\\
$70$ & $1.0000128440\dots$ & \\
$80$ & $1.0000634933\dots$ & \\
$90$ & $1.0004126231\dots$ & \\
$100$ & $1.0031219206\dots$ & \\
$110$ & $1.0256918168\dots$ & \\
$120$ & $1.5572034878\dots$ & \\
$130$ & $0.9163797290\dots$ & \\
\bottomrule
\end{tabular}
\end{table}

This failure demonstrates the difficulty of making predictions based on
limited numerical data.  If one looked at only the data for $k \le 40$ and
$n=8$, one might reasonably conjecture that the bound was converging to $1$
(although a sophisticated analysis would indicate that the convergence was
happening uncomfortably slowly as $k$ neared $40$).

This effect is reminiscent of Runge's phenomenon from interpolation theory (see \cite{Ep}).
Although the problem is not literally overconstrained, forcing too many roots
at the limiting locations constrains the function so much that it develops
undesired oscillations to compensate.  Pushing the larger roots towards
infinity seemingly relaxes the constraints, dampens the oscillations, and
allows convergence.

\begin{table}
\caption{Upper bounds for packing density using the modified vector
lengths as forced root locations.  Bounds are expressed as a multiple
of the density of the optimal lattice.  Note the contrast with Table~\ref{table:naive}.} \label{table:modified}
\centering
\begin{tabular}{ccccc}
\toprule
$k$ & \ \ & packing bound in $\R^{8}$ & \ \ & packing bound in $\R^{24}$\\
\midrule
$25$  && $1+2.013636284513588\ldots\times10^{-10}$ && $1+1.276838479911905\ldots\times10^{-6\phantom{0}}$\\
$50$  && $1+5.356893094673532\ldots\times10^{-16}$ && $1+4.112485306793651\ldots\times10^{-11}$\\
$75$  && $1+2.843270958834257\ldots\times10^{-20}$ && $1+1.034793038360603\ldots\times10^{-14}$\\
\midrule
$100$ && $1+6.131875484794015\ldots\times10^{-24}$ && $1+6.036832814830833\ldots\times10^{-18}$\\
$200$ && $1+7.957229644125821\ldots\times10^{-35}$ && $1+1.224810072437178\ldots\times10^{-27}$\\
$300$ && $1+8.043925729944741\ldots\times10^{-43}$ && $1+6.139675825632854\ldots\times10^{-35}$\\
$400$ && $1+1.554622153413999\ldots\times10^{-49}$ && $1+3.603565234648839\ldots\times10^{-41}$\\
$500$ && $1+4.477920519243749\ldots\times10^{-55}$ && $1+2.511348284489217\ldots\times10^{-46}$\\
$600$ && $1+6.319153710652842\ldots\times10^{-60}$ && $1+7.276989083620164\ldots\times10^{-51}$\\
\bottomrule
\end{tabular}
\end{table}

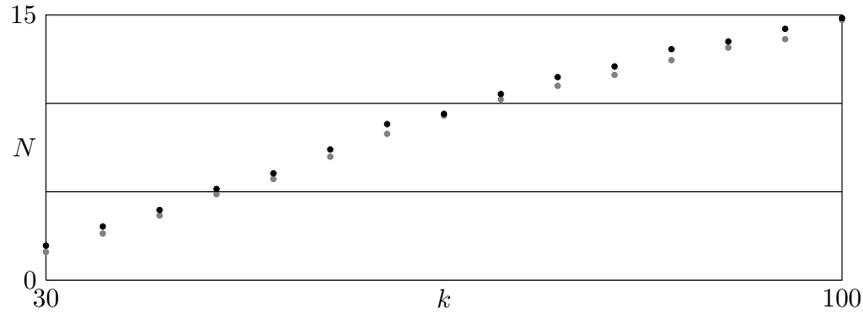
\begin{figure}
\centering
\begin{tikzpicture}[scale=0.0352777777]
\draw (0,33.3333) -- (300,33.3333);
\draw (0,66.6667) -- (300,66.6667);
\draw (0,0) -- (300,0) -- (300,100) -- (0,100) -- (0,0);
\draw (0,0) node[left] {$0$};
\draw (0,50) node[left] {$N$};
\draw (0,100) node[left] {$15$};
\draw (0,0) node[below] {$30$};
\draw (150,0) node[below] {$k$};
\draw (300,0) node[below] {$100$};
\foreach \x/\a/\b in
{0/10.655229/13.036089, 21.428571/17.613203/20.248447, 42.857143/24.378345/26.461763,
64.285714/32.413545/34.404995, 85.714286/38.146460/40.284130, 107.14286/46.554480/49.274673,
128.57143/55.190897/58.866695, 150.00000/62.066051/62.725017, 171.42857/68.164020/70.190967,
192.85714/73.303053/76.590207, 214.28571/77.411560/80.605387, 235.71429/82.968007/87.118780,
257.14286/87.692867/90.014060, 278.57143/90.893707/94.806940, 300.00000/98.101593/98.847573}
{ \draw[fill=gray, color=gray] (\x,\a) circle (1);
\draw[fill=black] (\x,\b) circle (1); }
\end{tikzpicture}
\caption{Number $N$ of digits to which the values of $f_k$ and
$\widehat{f}_k$ agree with those of $f_{k-5}$ and
$\widehat{f}_{k-5}$ at all of the points $x/10 + (y/10)i$, for integers $0 \le x \le 50$
and $0 \le y \le 2$.  Data points
for $n=8$ are gray and those for $n=24$ are black.}
\label{imagplotfigure}
\end{figure}

\begin{table}
\caption{Values of $f_k(i/2)$ and $\widehat{f}_k(i/2)$.}
\label{table:valuesonimagaxis}
\centering
\begin{tabular}{ccccc}
\toprule
$n$ & $k$ & $f_k(i/2)$ & $\widehat{f}_k(i/2)$\\
\midrule
$8$ & $100$ & $0.939432541969057457603843\dots$ & $0.526774741363446491086599\dots$\\
$8$ & $200$ & $0.939432541959478373173290\dots$ & $0.526774741373025575517211\dots$\\
$8$ & $300$ & $0.939432541959477686529550\dots$ & $0.526774741373026262160950\dots$\\
$8$ & $400$ & $0.939432541959477685343726\dots$ & $0.526774741373026263346775\dots$\\
$8$ & $500$ & $0.939432541959477685338723\dots$ & $0.526774741373026263351777\dots$\\
$8$ & $600$ & $0.939432541959477685338602\dots$ & $0.526774741373026263351898\dots$\\
\midrule
$24$ & $100$ & $0.909504018094605062955468\dots$ & $0.543934528596990605074180\dots$\\
$24$ & $200$ & $0.909504017149389039571803\dots$ & $0.543934529542206632888889\dots$\\
$24$ & $300$ & $0.909504017149302144551677\dots$ & $0.543934529542293527909015\dots$\\
$24$ & $400$ & $0.909504017149301977449339\dots$ & $0.543934529542293695011353\dots$\\
$24$ & $500$ & $0.909504017149301976704937\dots$ & $0.543934529542293695755754\dots$\\
$24$ & $600$ & $0.909504017149301976686050\dots$ & $0.543934529542293695774641\dots$\\
\bottomrule
\end{tabular}
\end{table}

We have been unable to analyze the asymptotic behavior of the functions $f_k$
defined using the roots \eqref{eq:betterrdef}, but they lead to excellent bounds (see
Table~\ref{table:modified}) and   appear to converge rapidly. In what
follows, we refer to $f_k$ and $\widehat{f_k}$, as well as their hypothetical
limits as $k \to \infty$, as analytic functions of a radial variable.

\begin{conjecture}
\label{conj:doubleconv} As $k \to \infty$, $f_k$ converges to a
function $f$ and $\widehat{f}_k$ converges to $\widehat{f}$, on
some neighborhood of the real line in $\C$. The convergence is
uniform on compact subsets of this neighborhood.
\end{conjecture}

The evidence for uniform convergence is of course not as strong as that for convergence,
but it implies that $f$ and $\widehat{f}$ are analytic and thus rapidly
decreasing (because their Fourier transforms are analytic and hence smooth).
It follows that they are admissible.

As evidence for Conjecture~\ref{conj:doubleconv}, we offer
Figure~\ref{imagplotfigure}, which demonstrates steady convergence as $k$
increases from $30$ to $100$, at a selection of sample points with real parts
up to $5$ and imaginary parts up to $0.2$.  In fact, convergence seems to
hold even for somewhat larger imaginary parts; for example,
Table~\ref{table:valuesonimagaxis} shows the values at $i/2$.  However,
convergence does not occur when the imaginary part is $1$ or more; for
example, for $n=8$ we have
\[
f_{500}(i) = 1786219116279967.87\ldots
\]
while
\[
f_{600}(i) = 474994401497433517.69\ldots.
\]

\begin{conjecture}
\label{conj:roots}
The limiting functions $f$ and $\widehat{f}$ from Conjecture~\ref{conj:doubleconv}
have no real roots other than the forced roots.
\end{conjecture}

When $n=24$, $\widehat{f_2}$ has another real root, but we have found no
other case in which $f_k$ or $\widehat{f_k}$ has any non-forced real roots.
If Conjecture~\ref{conj:doubleconv} holds and there is a neighborhood of the real
line into which the complex roots never intrude, then that is enough to imply
Conjecture~\ref{conj:roots}.  However, it is unclear whether this stronger
hypothesis is true.
As one can see from the data in Table~\ref{table:rootdistance}, the complex roots are growing steadily
closer to the real axis, and they might reach it around $k=1400$.  Even if
they eventually reach the axis, we conjecture that any unwanted
sign changes will occur far from the origin and will disappear in the limit
as $k \to \infty$.  It is plausible that one could remove them entirely by modifying
\eqref{eq:betterrdef}.

The complex root locations have several mysterious properties.  See
Figure~\ref{complexrootsfigure} for plots with $k=600$ and $n=8$, and
Figure~\ref{complexrootsfigure2} for plots with $n=24$ (which are very
similar to the $n=8$ case). The roots lie on several clear curves, and they
are most likely accumulating on the boundary of the domain of holomorphy.
Note that their nearest approach to the real axis is quite far from the
origin, as we asserted above.

One surprising observation is that $f_k$ and $\widehat{f}_k$ have nearly the
same roots away from the origin.  In the third part of these two figures, we
show the roots of one of $f_k$ or $\widehat{f}_k$ that do not agree to six
decimal places with any root of the other.  Only the roots relatively near
the origin appear in these plots.  See also Table~\ref{table:rootdistance},
in which the $f_k$ and $\widehat{f}_k$ columns become nearly identical as $k$
grows.

For comparison, Table~\ref{table:roots} shows the nearest roots to the
origin. In each case, $\widehat{f}_k$ has a purely imaginary root that is
probably converging as $k \to \infty$ (the numbers show clear convergence
when $n=8$ and possible convergence when $n=24$). The other roots are roughly
paired up for $f_k$ and $\widehat{f}_k$, but these pairs are not nearly as
close to each other as those further from the origin. We see no reason to
think any of the non-real roots are converging except for the purely
imaginary roots.

It is clear from this data that the roots have considerable structure, which
we are unable to explain conceptually.  More data could help, but
calculations for large $k$ are very time consuming.  We have computed $f_k$
and $\widehat{f}_k$ for $k=700$, $800$, and $900$, but we have not located
their roots.  If they have no unexpected sign changes, then with $k=900$ we
get sphere packing bounds within a factor of $1+5.33 \times 10^{-72}$ of the
density of $E_8$ or $1+3.04\times 10^{-62}$ of that of the Leech lattice.  We
expect that these bounds are true and could be proved given enough computing
power, but the evidence is not as conclusive as it is in the cases for which
we have located the roots.

\begin{table}
\caption{The minimal distance between the complex roots of $f_k$ or
$\widehat{f}_k$ and the real axis.} \label{table:rootdistance}
\centering

\caption{The non-real roots of $f_{600}$ (above) and $\widehat{f}_{600}$
(middle) in the right half-plane, for $n=8$.  The lower graph shows the points
in either of the previous two graphs for which no point in the other agrees
to within $10^{-6}$ in real and imaginary parts.}
\label{complexrootsfigure}
\end{figure}

\begin{figure}
%
\caption{The non-real roots of $f_{600}$ (above) and $\widehat{f}_{600}$
(middle) in the right half-plane, for $n=24$.  The lower graph shows the points
in either of the previous two graphs for which no point in the other agrees
to within $10^{-6}$ in real and imaginary parts.} \label{complexrootsfigure2}
\end{figure}

\begin{table}
\caption{The leftmost roots of $f_k$ and $\widehat{f}_k$ in the right
half-plane.  Each of the four parts of the table corresponds to the values of
$n$ and $k$ specified in square brackets.} \label{table:roots}
\centering
\begin{tabular}{ccc}
\toprule
roots of $f_k$ &\ \ \ & roots of $\widehat{f}_k$\\
\midrule
{[}$n=8$, $k=500$] && $\pm 0.6817374606\ldots i$\\
$0.01864424055\ldots \pm 0.7968630734\ldots i$ &&
$0.01969453528\ldots \pm 0.7922341309\ldots i$\\
$0.05589098197\ldots \pm 0.7966856319\ldots i$ &&
$0.05879879542\ldots \pm 0.7927998039\ldots i$\\
\midrule
{[}$n=8$, $k=600$] && $\pm 0.6817374605\ldots i$\\
$0.01690822801\ldots \pm 0.7854640419\ldots i$ &&
$0.01788437319\ldots \pm 0.7806599882\ldots i$\\
$0.05069999170\ldots \pm 0.7853472831\ldots i$ &&
$0.05339644678\ldots \pm 0.7812202661\ldots i$\\
\midrule
{[}$n=24$, $k=500$] && $\pm 0.7236064057\ldots i$\\
$0.01960883579\ldots \pm 0.7855031742\ldots i$ &&
$0.02223833238\ldots \pm 0.7756885639\ldots i$\\
$0.05877932631\ldots \pm 0.7854671622\ldots i$ &&
$0.06451464105\ldots \pm 0.7782266879\ldots i$\\
\midrule
{[}$n=24$, $k=600$] && $\pm 0.7235866774\ldots i$\\
$0.01772546664\ldots \pm 0.7746247816\ldots i$ &&
$0.02044183069\ldots \pm 0.7644882404\ldots i$\\
$0.05314786265\ldots \pm 0.7746157584\ldots i$ &&
$0.05868760933\ldots \pm 0.7671872439\ldots i$\\
\bottomrule
\end{tabular}
\end{table}

It follows from Conjecture~\ref{conj:roots} that $f$ and
$\widehat{f}$ have no unexpected sign changes.  Thus,
Conjectures~\ref{conj:doubleconv} and~\ref{conj:roots} for $n=8$
or $24$ would solve the sphere packing problem in $\R^n$.

It is interesting to note that the parameter $n$ in these
conjectures can be varied, while leaving the forced root
locations fixed. Of course there is no connection with sphere
packing for general $n$ (it does not even have to be an
integer). If a limiting $f$ exists, it also does not follow in
general that $f(0)=\widehat{f}(0)$, since that requires Poisson
summation over an appropriate lattice.  However, the analogues
of Conjectures~\ref{conj:doubleconv} and~\ref{conj:roots} do
seem to hold in all small dimensions (although we have not
investigated them as carefully as the $n=8$ and $n=24$ cases).  In particular, we
conjecture that if $f_k$ is defined with forced roots based on
the $E_8$ vector lengths, then these conjectures hold for $0 < n < 10$ (for $n=10$ there in fact appear to be extraneous
real roots). This flexibility is encouraging, because it
suggests that a proof need not depend on specific facts about
$\R^8$, but rather could hold for much more general reasons.
Similarly, for the Leech lattice vector lengths the conjectures
seem to hold for $0 < n < 26$.  More generally,
many of the phenomena we study in this paper are
not restricted to $n=8$ and $n=24$.  For example, we make the following
conjecture:

\begin{conjecture}
For $0 < n < 10$, forcing roots at the $E_8$ vector lengths yields a limiting
function $f$ satisfying
\[
\frac{f(0)}{\widehat{f}(0)} = -\frac{n^4-56n^3+1184n^2-11200n+40320}{16(n-10)(n-14)(n-18)}.
\]
For $0 < n < 26$, using the Leech lattice vector lengths yields instead
\[
\frac{f(0)}{\widehat{f}(0)} = -\frac{p_{24}(n)}{32(n-26)(n-34)(n-38)(n-42)(n^3
-116n^2+4480n-57024)},
\]
where
\begin{align*}
p_{24}(n) &= n^8-284n^7+35312n^6-2510720n^5+111652352n^4-3180064256n^3\\
& \quad \phantom{} + 56651266048n^2-577142292480n+2574499479552.
\end{align*}
\end{conjecture}

This conjecture is evidence that the limiting functions have even more
intricate structure than is apparent just from the $n=8$ and $n=24$ cases.

Note that for reasons of computational efficiency, one should never
solve the equations~\eqref{constreqn} directly. Instead,
it is more convenient to solve two systems, each of half the size.  To form them, we write the
polynomial $q_k$ from the definition $q_k(|x|^2) e^{-\pi|x|^2}$ of $f_k(x)$
as the sum $q_k^0 + q_k^1$,
where
\[
q_k^\varepsilon = \frac{q_k + (-1)^\varepsilon \T q_k}{2}
\]
for $\varepsilon \in \{0,1\}$. Then
\[
\T q_k^\varepsilon =
(-1)^\varepsilon q_k^\varepsilon.
\]
(In other words, we have diagonalized the Fourier transform.)

We can express $q_k^0$ as a linear combination of the rescaled Laguerre polynomials
$p_j$ with $j$ even, and $q_k^1$ as a linear combination with $j$
odd. The constraints on $f_k$ and $\widehat{f}_k$ amount to the
following individual constraints on $q_k^\varepsilon$:
\begin{equation}
\label{const2a} \textup{$q_k^\varepsilon$ vanishes to order $1$ at
$r_1^2$ and order $2$ at $r_2^2, r_3^2,\dots, r_k^2$.}
\end{equation}
The only missing constraint is that $\widehat{f_k}$ must have a
double root at $r_1$ (\eqref{const2a} forces only a single root).
The issue is that given only the constraints above, $q_k^0$ and
$q_k^1$ are only determined up to scaling, and may be scaled
independently; to produce the double root the scalings must be
compatible.

The following determinant gives a formula for
$q_k^\varepsilon(x)$, up to scaling (it follows using the approach
of Lemma~\ref{unique} that this determinant is not identically
zero):
\[
\left|\begin{matrix}
p_\varepsilon(x) & p_{2+\varepsilon}(x) & p_{4+\varepsilon}(x) & \cdots &  p_{4k-2+\varepsilon}(x) \\
p_\varepsilon(r_1^2) & p_{2+\varepsilon}(r_1^2) & p_{4+\varepsilon}(r_1^2) & \cdots &  p_{4k-2+\varepsilon}(r_1^2) \\
p_\varepsilon(r_2^2) & p_{2+\varepsilon}(r_2^2) & p_{4+\varepsilon}(r_2^2) & \cdots &  p_{4k-2+\varepsilon}(r_2^2) \\
p_\varepsilon'(r_2^2) & p_{2+\varepsilon}'(r_2^2) & p_{4+\varepsilon}'(r_2^2) & \cdots &  p_{4k-2+\varepsilon}'(r_2^2) \\
p_\varepsilon(r_3^2) & p_{2+\varepsilon}(r_3^2) & p_{4+\varepsilon}(r_3^2) & \cdots &  p_{4k-2+\varepsilon}(r_3^2) \\
p_\varepsilon'(r_3^2) & p_{2+\varepsilon}'(r_3^2) & p_{4+\varepsilon}'(r_3^2) & \cdots &  p_{4k-2+\varepsilon}'(r_3^2) \\
& & \vdots & & \\
p_\varepsilon(r_k^2) & p_{2+\varepsilon}(r_k^2) & p_{4+\varepsilon}(r_k^2) & \cdots &  p_{4k-2+\varepsilon}(r_k^2) \\
p_\varepsilon'(r_k^2) & p_{2+\varepsilon}'(r_k^2) & p_{4+\varepsilon}'(r_k^2) & \cdots &  p_{4k-2+\varepsilon}'(r_k^2) \\
\end{matrix}\right|
\]
It is tempting to take the limit as $k \to \infty$ and hope to
write down an infinite determinant for the limiting function.
However, we see no way to make sense of this idea.

Computing $q_k^0$ and $q_k^1$ independently is substantially faster than
computing $q_k$ (approximately four times faster using a cubic-time
algorithm). So far, it has not led to any theoretical advances, but in
Section~\ref{sec:single} we will see a closely related example in which it is
theoretically important to separate the Fourier eigenfunctions.

\section{Rationality}
\label{sec:rat}

Although we are unable to identify the proposed limiting functions $f$ for
dimensions 8 and 24, we can say two things about their special values.
In fact, the analysis we provide applies to the functions in the statement of
Conjecture~\ref{conjecture:cohnelkies},
and in particular the function explicitly exhibited in \cite{V} for the $n=8$ case.
The first is a property we can derive, while the second has been observed only
numerically and so far lacks an explanation.

The first observation is that we can predict the value of
$f'(r_1)$, where $r_1$ is the first forced root.  Here we view
$f$ as a function of a single radial variable, so $f'$ is the
radial derivative. By condition \eqref{constraints}, knowing
$f'(r_1)$ means we know the values of both $f$ and $f'$ at
every vector length in the respective lattices ($E_8$ and
Leech) for dimensions 8 and 24.

\begin{lemma}
\label{fprime} Let $n \in \{2,8,24\}$, and let $f$ be a hypothetical optimal
function for use in Theorem~\ref{theorem:cohnelkies}, as in
Conjecture~\ref{conjecture:cohnelkies}. Then
\begin{equation}\label{fprimeformula}
    f'(r_1) = -\frac{n}{Nr_1}\widehat{f}(0),
\end{equation}
where
\[
r_1 = \text{the minimal vector length} =
\begin{cases}
(4/3)^{1/4} & \textup{for $n=2$,}\\
\sqrt{2} & \textup{for $n=8$, and}\\
2 & \textup{for $n=24$}\\
\end{cases}
\]
and
\[
    N = \text{the number of minimal vectors} = \begin{cases}
    6 & \textup{for $n = 2$,}\\
    240  &   \textup{for $n  =  8$, and} \\
    196560   &  \textup{for $n  =  24$.} \\
\end{cases}
\]
\end{lemma}

Note that without loss of generality, we assume that $f$ is radial.

\begin{proof}
Define rescaled functions for $r>0$ by
\[
    f_r(x) = f(rx) \qquad \textup{and} \qquad
    \widehat{f}_r(t)     = r^{-n}  \widehat{f}\left({
    {t}/{r}}\right).
\]
Let
\[
    F(x) =  \left. \frac{d}{dr}\right|_{r=1}  f_r(x)     = |x|f'(x),
\]
so that
\[
    \widehat{F}(t) =  - n  \widehat{f}(t) -     t\widehat{f}\; '(t).
\]
Now apply Poisson summation to $F$ over optimal lattice $\Lambda$.
Removing terms where $F$ or $\widehat{F}$ is forced to vanish,
this identity states
\[
    \sum_{x \in \Lambda, \,|x|=r_1} |x|f'(x) = -n\widehat{f}(0),
\]
which is \eqref{fprimeformula}.
\end{proof}

The second---and perhaps more interesting---feature we have noticed is that the Taylor series
for $f$ and $\widehat{f}$, normalized so that
$f(0)=\widehat{f}(0)=1$, have rational quadratic coefficients.
Table~\ref{table:coeffs} shows numerical
evidence for this.  It displays the second and fourth Taylor
coefficients for $f$ and $\widehat{f}$ in dimensions $8$ and
$24$.  (We cannot be certain that all the reported digits are
correct for the limiting functions, but they agree for $k=300$
and $k=600$.) One can see from the decimal expansions that the
quadratic coefficients are rational, but the quartic
coefficients remain mysterious.

\begin{table}
\caption{Approximate Taylor series coefficients of $f$ and
$\widehat{f}$ about $0$.} \label{table:coeffs}
\begin{tabular}{ccccc}
\toprule
$n$ & function & order & coefficient & conjecture\\
\midrule $8$ & $f$ & $2$ & $-2.7000000000000000000000000000\dots$ & $-27/10$\\
$8$ & $\widehat{f}$ & $2$ & $-1.5000000000000000000000000000\dots$ & $-3/2$\\
$24$ & $f$ & $2$ & $-2.6276556776556776556776556776\dots$ & $-14347/5460$\\
$24$ & $\widehat{f}$ & $2$ & $-1.3141025641025641025641025641\dots$ & $-205/156$\\
$8$ & $f$ & $4$ & $\phantom{-}4.2167501240968298210998965628\dots$ & ?\\
$8$ & $\widehat{f}$ & $4$ & $-1.2397969070295980026220596589\dots$ & ?\\
$24$ & $f$ & $4$ & $\phantom{-}3.8619903167183007758184168473\dots$ & ?\\
$24$ & $\widehat{f}$ & $4$ & $-0.7376727789015322303799539712\dots$ & ?\\
\bottomrule
\end{tabular}
\end{table}

\begin{conjecture} \label{conjecture:rational}
For $n=8$, the limiting functions $f$ and $\widehat{f}$ have quadratic Taylor
coefficients $-27/10$ and $-3/2$, respectively (when normalized so that
$f(0)=\widehat{f}(0)=1$). For $n=24$, the corresponding coefficients are
$-14347/5460$ and $-205/156$.
\end{conjecture}

The same is true when $n=8$ for the functions studied in \cite{V}.

We do not know whether the higher Taylor coefficients are rational or even
given by simple expressions at all.  Needless to say, it would be interesting
to have explicit formulas for the general coefficients, because this would
give a direct construction of $f$ and $\widehat{f}$ by power series and
analytic continuation.

To put this conjecture in a slightly broader context, consider the
Mellin transform
\[
M_f(s) = \int_0^\infty f(x) x^{s-1} \, dx.
\]
When $f$ is smooth and rapidly decreasing (as
Conjecture~\ref{conj:doubleconv} implies), the integral converges to a
holomorphic function for $\Re{s}>0$.  It is a standard fact that $M_f(s)$
can be meromorphically continued to $\C$, with at most simple poles at $s\in
\Z_{\le 0}$; furthermore, for integers  $j\ge 0$ its residue at $s=-j$ is
the $j$-th Taylor coefficient of $f$. To see why, note that if $f(x)$ has
the Taylor series expansion $\sum_{j \ge 0} c_j x^j$ about $x=0$, then
\[
M_f(s) = \int_0^1 \left(f(x) - \sum_{j=0}^\ell c_j x^j\right)x^{s-1} \, dx +
\sum_{j=0}^\ell \frac{c_j}{s+j} + \int_1^\infty f(x) x^{s-1} \, dx,
\]
where both integrals converge as long as $\Re{s}>-(\ell+1)$.  Since our function $f$ is radial, its Taylor coefficients $c_j$ vanish if $j$ is odd.

A short calculation (see \cite[Theorem~5.9]{LL}) shows that if $\widehat{f}$ is the $n$-dimensional
Fourier transform of $f$ (interpreted as a radial function), then
\begin{equation} \label{eq:mellin-symm}
M_{\widehat{f}}(s) = \frac{\pi^{n/2-s}\Gamma(s/2)}{\Gamma((n-s)/2)} M_f(n-s),
\end{equation}
valid as an identity of meromorphic functions on $\C$.
In particular, computing the residue of $M_f(s)$ at $s=-j$
shows that the $j$-th Taylor coefficient of $f$ equals
\[
(-1)^{j/2} \frac{2\pi^{j+n/2}}{\Gamma(j/2+1)\Gamma((n+j)/2)} M_{\widehat{f}}(n+j),
\]
and vice versa with $f$ and $\widehat{f}$ switched.

Thus, in $n$ dimensions Conjecture~\ref{conjecture:rational} amounts to
specifying $M_f(n+2)$ and $M_{\widehat{f}}(n+2)$.  The values $M_f(n)$ and
$M_{\widehat{f}}(n)$ are easy consequences of $f(0)=\widehat{f}(0)=1$.  We
have identified one other value, namely the midpoint $4$ of the $s
\leftrightarrow n-s$ symmetry when $n=8$:

\begin{conjecture} \label{conjecture:midvalue}
For $n=8$, the limiting functions satisfy
\[
M_f(4) = M_{\widehat{f}}(4) = \frac{1}{15}
\]
when normalized with $f(0)=\widehat{f}(0)=1$.
\end{conjecture}

The equality $M_f(4) = M_{\widehat{f}}(4)$ follows from
\eqref{eq:mellin-symm}, but not the value $1/15$.  It is natural to expect a
corresponding conjecture for $n=24$, but we have been unable to identify the
numerical value
\[
M_f(12) = M_{\widehat{f}}(12) = 0.177860964729650276645646126241\ldots
\]
in that case.

\section{Energy minimization}
\label{sec:energy}

One natural generalization of sphere packing is potential energy
minimization.  Given a radial potential function $\varphi \colon \R^n \to \R$
and a set $\mathcal{P}$ of point particles, the energy
$E_\varphi(\mathcal{P},x)$ of a particle $x \in \mathcal{P}$ is
defined to be
\[
\sum_{y \in \mathcal{P},\ y \neq x} \varphi(x-y),
\]
and the energy $E_\varphi(\mathcal{P})$ is defined as the average of
$E_\varphi(\mathcal{P},x)$ over all $x \in \mathcal{P}$.  (Of course some
hypotheses are needed for this to make sense, but it is well defined when
$\mathcal{P}$ is a periodic discrete set and $\varphi$ is rapidly
decreasing.) The question of  how to choose $\mathcal{P}$ so as to minimize
energy with a fixed density, arises naturally in physics; see \cite{C2} for
a survey.

Cohn and Kumar \cite{CK2} defined a configuration $\mathcal{P}$ to be
\emph{universally optimal} if it minimizes energy whenever $\varphi(x)$ is
completely monotonic as a function of $|x|^2$ and decreases sufficiently
quickly.  For example, $\varphi$ could be a sufficiently steep inverse power
law. As explained in Section~9 of \cite{CK2}, it suffices to check optimality
for the Gaussians $\varphi(x) = e^{-c|x|^2}$ with $c>0$, i.e.,
 the Gaussian core model \cite{S} from mathematical physics. Cohn and
Kumar conjectured that the hexagonal lattice, $E_8$, and the Leech lattice
are universally optimal.  (See \cite{CKS} for information about ground states
in other dimensions.)

Proposition~9.3 of \cite{CK2} offers an approach to proving this conjecture
by linear programming bounds, which Cohn and Kumar conjectured were sharp in
these special dimensions (much like the case of sphere packing).  Given an
admissible auxiliary function $h \colon \R^n \to \R$ satisfying $h \le
\varphi$ and $\widehat{h} \ge 0$ everywhere, this proposition says that every
configuration $\mathcal{P}$ of density $1$ satisfies
\[
E_\varphi(\mathcal{P}) \ge \widehat{h}(0)-h(0).
\]

We can construct $h$ by imitating the sphere packing construction: let $h(x)$
be a radial polynomial times $e^{-\pi|x|^2}$, with the polynomial chosen with
minimal degree so that $\varphi-h$ and $\widehat{h}$ have double roots at the
modified root locations from Section~\ref{sec:conj}.  We conjecture that as
the number of roots tends to infinity, these functions converge and the
limiting functions prove a sharp bound for energy.

The closest analogue of Conjectures~\ref{conjecture:rational}
and~\ref{conjecture:midvalue} we have found is the following.

\begin{conjecture} \label{conjecture:energy}
For the potential function $\varphi(x) = e^{-c|x|^2}$ in $\R^n$ with $n=8$ or
$24$, the limiting auxiliary function $h$ satisfies
\[
\widehat{h}(0) = \frac{2c}{n} E_\psi(\Lambda_n),
\]
where $\Lambda_n$ is $E_8$ or the Leech lattice when $n=8$ or $24$,
respectively, and $\psi(x) = |x|^2\varphi(x)$.
\end{conjecture}

Besides numerical evidence, one reason to believe this conjecture is that it
is compatible with duality symmetry.  If the auxiliary function $h$ proves an
energy bound for an integrable potential function $\varphi$, then $g :=
\widehat{\varphi}-\widehat{h}$ does so for $\widehat{\varphi}$. Specifically,
$g \le \widehat{\varphi}$ since $\widehat{h} \ge 0$, and $\widehat{g} \ge 0$
since $h \le \varphi$.  This duality transformation preserves optimality: if
$h$ proves that a lattice $\Lambda$ of covolume $1$ minimizes $E_\varphi$,
then $g$ proves that the dual lattice $\Lambda^*$ minimizes
$E_{\widehat{\varphi}}$.  To see why, note that $\varphi(0) + \widehat{h}(0)
- h(0) = \widehat{\varphi}(0) + \widehat{g}(0) - g(0)$, from which it follows
by Poisson summation that
\[
\widehat{h}(0) - h(0) = E_\varphi(\Lambda)
\]
if and only if
\[
\widehat{g}(0) - g(0) = E_{\widehat{\varphi}}(\Lambda^*).
\]
This duality is compatible with Conjecture~\ref{conjecture:energy}, in the
sense that $h$ satisfies the conjecture if and only if $g$ does; the
compatibility is not obvious, but it follows from a short calculation using
Poisson summation.

\section{Forcing single roots}
\label{sec:single}

In this section we discuss a related problem: constructing
functions with forced single roots (instead of the forced
double roots used earlier in the paper). Such functions do not
have direct applications to sphere packing, but they can be
explicitly written down in some cases and thus serve as a
testing ground for ideas concerning our main conjectures.
Furthermore, their properties are quite a bit more interesting
than one would guess from their definition.

The structure in this problem is best seen by forcing single roots for
Fourier eigenfunctions.  The use of eigenfunctions was merely a computational
convenience in Section~\ref{sec:num}, but in this section it will play an
essential role in our conjectures.

For $\varepsilon \in \{0,1\}$, let
\[
g^\varepsilon_{n,k} = r^\varepsilon_{n,k}(|x|^2) e^{-\pi|x|^2},
\]
where $r^\varepsilon_{n,k}$ is a polynomial of degree at most
$2k+\varepsilon$ that is not identically zero, vanishes at
$2,4,\dots,2k$, and is a linear combination of the polynomials
$p_{2j+\varepsilon}^{n/2-1}$ for $0 \le j \le k$.  The last
condition means that $\widehat{g}^\varepsilon_{n,k} =
(-1)^\varepsilon g^\varepsilon_{n,k}$.  The same arguments as in
Lemma~\ref{unique} shows that these functions exist and are
unique, up to scaling. We see no canonical way to scale them, so
we will not choose a preferred scaling.

The choice of $2,4,\dots,2k$ as forced root locations is inspired by the
norms of the vectors in the $E_8$ lattice.  One could also study the
analogous functions for the Leech lattice, but we have focused on the
simplest case.  Note that we use the exact vector norms, with no need to
modify them along the lines of \eqref{eq:betterrdef}.

\begin{conjecture}
\label{conj:convergence} As $k \to \infty$ with $n$ and
$\varepsilon$ fixed, $g^\varepsilon_{n,k}$ converges (when
suitably normalized) to a Fourier eigenfunction
$g^\varepsilon_{n}$ (not identically zero) that vanishes at all
radii of the form $\sqrt{2j}$.  If we view $g^\varepsilon_{n,k}$
as an entire function of $|x|$, then the convergence is uniform on
all compact subsets of $\C$.
\end{conjecture}

Uniform convergence implies that $g^\varepsilon_{n}(x)$ is an entire function
of $|x|$.

These limiting functions are mysterious in general, but when $n$
is a multiple of $4$ we can conjecture explicit formulas for half
of them.  The remaining functions appear to be much more subtle,
as we will see shortly.

\begin{conjecture}
\label{conj:singleforce} If the scaling is chosen appropriately,
then
\[
g^0_4(x) = {\frac{\sin \pi |x|^2/2}{\pi |x|^2/2}}
e^{-\pi\sqrt{3}|x|^2/2}.
\]
If $n>4$ is a multiple of $4$ and $\varepsilon \not\equiv n/4
\pmod {2}$, then (again up to scaling)
\[
g^\varepsilon_{n}(x) = \left(\sin \pi |x|^2/2\right) e^{-\pi
\sqrt{3} |x|^2/2}
\]
if $n \equiv 0 \pmod{3}$,
\[
g^\varepsilon_{n}(x) = \left(\sin \pi |x|^2/2\right)
\left(\left(|x|^2 - \frac{(n+2)\sqrt{3}}{6\pi}\right)^2 -
\frac{n+2}{6\pi^2}\right) e^{-\pi \sqrt{3} |x|^2/2}
\]
if $n \equiv 1 \pmod{3}$, and
\[
g^\varepsilon_{n}(x) = \left(\sin \pi |x|^2/2\right)
\left(|x|^2-n/(2\pi\sqrt{3})\right) e^{-\pi \sqrt{3} |x|^2/2}
\]
if $n \equiv 2 \pmod{3}$.
\end{conjecture}

We have no explanation for the exceptional behavior in four dimensions.

\begin{proposition}
The functions listed in Conjecture~\ref{conj:singleforce} are all
eigenfunctions of the Fourier transform, with the appropriate
eigenvalues.
\end{proposition}

\begin{proof}[Sketch of proof]
This can be verified by straightforward calculation.  Because $x
\mapsto e^{-\pi |x|^2}$ is its own Fourier transform, it follows
that when $\mathop{\textup{Re}}(\alpha) > 0$, the Fourier
transform of $x \mapsto e^{-\pi |x|^2 \alpha}$ is $x \mapsto
e^{-\pi |x|^2/\alpha}/\alpha^{n/2}$.  Differentiating with respect
to $\alpha$ allows one to compute the Fourier transform of $x
\mapsto |x|^{2k} e^{-\pi |x|^2 \alpha}$ for $k \in \{1,2,\dots\}$.
Finally, we write
\[
\left(\sin \pi |x|^2/2\right) e^{-\pi \sqrt{3} |x|^2/2} =
\frac{e^{-\pi|x|^2 (\sqrt{3}/{2}-{i}/{2})} -
e^{-\pi |x|^2 ({\sqrt{3}}/{2}+{i}/{2})}}{2i}.
\]
The result when $n \equiv 0 \pmod{3}$ follows easily from the fact that
$\sqrt{3}/2 + i/2$ is a $12$-th root of unity, and the results when $n
\not\equiv 0 \pmod{3}$ follow from similar but slightly more elaborate
calculations.  The trickiest case is when $n=4$, because it involves dividing
by $|x|^2$.  That can be handled by integrating with respect to $\alpha$
instead of differentiating.
\end{proof}

Note that the limiting functions in Conjecture~\ref{conj:singleforce}
are entire and have imaginary roots at $\sqrt{-2j}$ for each $j > 0$.
It is not clear where these imaginary roots come from.  There are also
a finite number of extraneous real roots, which appear to be needed to
create Fourier eigenfunctions.  If one plots the roots of these
functions  (as in Figure~\ref{complexrootssingle}), one sees the expected roots on the real axis, the surprising
purely imaginary roots, the finite set of extraneous roots, and a
V-shaped collection of non-real roots spreading out from the imaginary
axis.  It seems that as $k \to
\infty$, that final collection tends to infinity and contributes no
roots in the limit.  It does serve, however, to reduce the exponent in
the Gaussian from the original $-\pi$ to $-\pi\sqrt{3}/2$.

\begin{figure}
\centering
\begin{tikzpicture}[scale=0.75]
\draw (-6.5,0) -- (6.5,0);
\draw (0,-3.5) -- (0,3.5);
\foreach \x/\y in {0.828399/-2.95336, 0.828399/2.95336, 1.10293/-3.03665, 1.10293/3.03665,
0.571941/-2.87388, 0.571941/2.87388, 1.40232/-3.12556, 1.40232/3.12556,
0.329663/-2.79707, 0.329663/2.79707, 0.103657/-2.72612,
0.103657/2.72612, 1.73981/-3.22363, 1.73981/3.22363, 2.15019/-3.34023,
2.15019/3.34023, 0/2.44936, 0/2.00000, 0/1.41421, 0/0, 0.857387/0,
1.41421/0, 2.00000/0, 2.44949/0, 2.82843/0, 3.16228/0, 3.46410/0,
3.74166/0, 4.00000/0, 4.24264/0, 4.47214/0, 4.69042/0, 4.89898/0,
5.09902/0, 5.29150/0, 5.47723/0, 5.65685/0, 5.83095/0, 6.00000/0,
6.16441/0, 6.32456/0, -0.828399/2.95336, -0.828399/-2.95336,
-1.10293/3.03665, -1.10293/-3.03665, -0.571941/2.87388,
-0.571941/-2.87388, -1.40232/3.12556, -1.40232/-3.12556,
-0.329663/2.79707, -0.329663/-2.79707, -0.103657/2.72612,
-0.103657/-2.72612, -1.73981/3.22363, -1.73981/-3.22363,
-2.15019/3.34023, -2.15019/-3.34023, 0/-2.44936, 0/-2.00000,
0/-1.41421, 0/0, -0.857387/0, -1.41421/0, -2.00000/0, -2.44949/0,
-2.82843/0, -3.16228/0, -3.46410/0, -3.74166/0, -4.00000/0, -4.24264/0,
-4.47214/0, -4.69042/0, -4.89898/0, -5.09902/0, -5.29150/0, -5.47723/0,
-5.65685/0, -5.83095/0, -6.00000/0, -6.16441/0, -6.32456/0}
{
\draw[fill=black] (\x,\y) circle (0.03);
}
\foreach \i in {-6,-5,-4,-3,-2,-1,0,1,2,3,4,5,6}
{
\draw (\i,-0.15) -- (\i,0.15);
}
\foreach \i in {-3,-2,-1,0,1,2,3}
{
\draw (-0.15,\i) -- (0.15,\i);
}
\end{tikzpicture}
\caption{The complex roots of $g^1_{8,20}$.}
\label{complexrootssingle}
\end{figure}
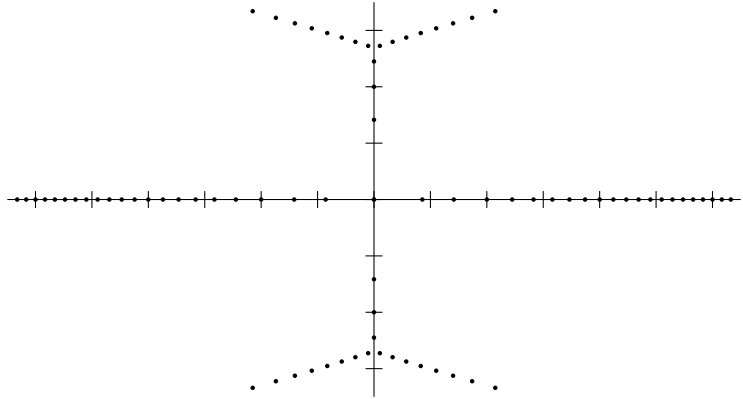

Note also that the root of $g^1_{8m}$ at the origin must occur by
Poisson summation over $E_8^m$ (so it is not extraneous in the
same sense, even though it was not deliberately forced).

Whenever the dimension is a multiple of $4$,
Conjecture~\ref{conj:singleforce} predicts one of the two eigenfunctions
Conjecture~\ref{conj:convergence} asserts exists. The other eigenfunction is more
mysterious.  It does not have imaginary roots at $\sqrt{-2j}$.  Instead, it
\emph{almost} has roots at $\sqrt{-(2j-1)}$, but not quite.  For example,
$g^0_8$ appears to vanish at the square roots of the numbers from
Table~\ref{table:g8roots}. The precise perturbations away from the odd
integers depend on the dimension.  We do not know an explanation for this
interesting behavior.  The most natural possibility is that these functions
are given by a dominant term, which has roots at exactly $\sqrt{-(2j-1)}$,
plus some lower order terms.  However, we have been unable to conjecture a
formula of this sort.

\begin{table}
\caption{Squares of the imaginary roots of $g^0_8$.}
\label{table:g8roots}
\begin{tabular}{r}
\toprule
$-0.980217784819734913\dots$\\
$-2.999513816437548808\dots$\\
$-4.999987267218782800\dots$\\
$-6.999999651348489332\dots$\\
$-8.999999990471834691\dots$\\
$-10.999999999741389569\dots$\\
$-12.999999999993001822\dots$\\
$-14.999999999999810493\dots$\\
$-16.999999999999994854\dots$\\
$-18.999999999999999859\dots$\\
$-20.999999999999999996\dots$\\
$-22.999999999999999999\dots$\\
\bottomrule
\end{tabular}
\end{table}

More general convergence theorems are likely true.  For example,
in four dimensions, if we force an extra factor of $|x|^2-c$ in
the $+1$ eigenfunction, then the resulting functions seem to
converge to
\[
g^0_4(x) \left(|x|^2-c\right) \left(\pi^2|x|^4 +
(c\pi^2-2\pi\sqrt{3})|x|^2+c(c\pi^2-2\pi\sqrt{3})\right).
\]
(Note that now the extraneous roots need not be real.) It seems
plausible that in each case covered by
Conjecture~\ref{conj:singleforce}, forcing some additional finite
set of roots simply creates an additional factor corresponding to
a finite set of extraneous roots.  That does not seem to be true
for the mysterious eigenfunctions not predicted by Conjecture~\ref{conj:singleforce} (i.e., forcing another root does
not seem to multiply them by a polynomial factor, but instead
changes the imaginary root perturbations as well).

In dimensions that are not multiples of four, we do not know any
closed form expressions for the limiting eigenfunctions.  It
appears that they behave somewhat like the mysterious
eigenfunctions in the multiple-of-four case (i.e., those not predicted by Conjecture~\ref{conj:singleforce}). For example, $g^0_2$
appears to have imaginary roots at some perturbation of the square
roots of $-2.5$, $-4.5$, $-6.5$, etc., and $g^1_2$ at some
perturbation of the square roots of $-1.5$, $-3.5$, $-5.5$, etc.
We have focused on the multiples of four because they seem simpler
(and perhaps more relevant to sphere packing).

\section*{Acknowledgements}

We are grateful to Percy Deift, Noam Elkies, Doug Hardin, Abhinav Kumar, Greg Moore, Ed
Saff, Alex Samorodnitsky, Terence Tao, and Frank Vallentin for helpful
conversations.

\end{document}